\documentclass{amsart}
\usepackage{latexsym,amssymb,amsmath,amsthm,amscd,graphicx}
\usepackage{longtable}
\usepackage{enumerate}

\numberwithin{equation}{section}

\theoremstyle{plain}
\newtheorem{thm}{Theorem}[section]
\newtheorem*{thm*}{Theorem}

\theoremstyle{definition}

\theoremstyle{remark}

\newtheorem*{xrem}{Remark}

\newcommand{\cD}{{\mathcal D}}

\newcommand{\cG}{{\mathcal G}}

\newcommand{\cM}{{\mathcal M}}

\newcommand{\cQ}{{\mathcal Q}}

\newcommand{\cS}{{\mathcal S}}

\newcommand{\kI}{{\mathfrak I}}

\newcommand{\kM}{{\mathfrak M}}

\newcommand{\N}{{\mathbb N}}

\newcommand{\R}{{\mathbb R}}

\newcommand{\Z}{{\mathbb Z}}

\def\al{\alpha}
\def\bt{\beta}

\def\tht{\theta}

\def\0{\emptyset}
\def\1{{\bf 1}}
\def\6{\partial}
\def\8{\infty}

\def\lt{\left}
\def\rt{\right}

\newcommand{\iii}[1]{{\left\vert\kern-0.25ex\left\vert\kern-0.25ex\left\vert #1 
    \right\vert\kern-0.25ex\right\vert\kern-0.25ex\right\vert}}

 \def\XXint#1#2#3{{\setbox0=\hbox{$#1{#2#3}{\int}$}
 \vcenter{\hbox{$#2#3$}}\kern-.5\wd0}}

\begin{document}

\title{The Adams trace theorem for product Morrey spaces}

\author[N.~Hatano]{Naoya Hatano}
\address{
Graduate School of Information Science and Technology, The University of Osaka, 1-5, Yamadaoka, Suita-shi, Osaka 565-0871, Japan
}
\email{n.hatano.chuo@gmail.com}

\author[R.~Kawasumi]{Ryota Kawasumi}
\address{
Center for Mathematics and Data Science, Gunma University, 
4-2 Aramaki-machi, Maebashi City, Gunma 371-8510, Japan 
}
\email{r-kawasumi@gunma-u.ac.jp}

\author[H.~Saito]{Hiroki Saito}
\address{
College of Science and Technology, Nihon University,
Narashinodai 7-24-1, Funabashi City, Chiba, 274-8501, Japan
}
\email{saitou.hiroki@nihon-u.ac.jp}

\author[H.~Tanaka]{Hitoshi Tanaka}
\address{
Research and Support Center on Higher Education for the hearing and Visually Impaired, 
National University Corporation Tsukuba University of Technology,
Kasuga 4-12-7, Tsukuba City, Ibaraki, 305-8521 Japan
}
\email{htanaka@k.tsukuba-tech.ac.jp}

\thanks{}



\date{}

\begin{abstract}
By using a~Hedberg-type inequality, 
the Adams trace inequality is extended 
from Lebesgue spaces to product Morrey spaces.
\end{abstract}

\maketitle

\section{Introduction}\label{sec1}
In this paper 
we investigate the Adams trace inequality for product Morrey spaces.
Let $\R^n$ be the classical $n$-dimensional Euclidean space. 
The \textit{fractional integral operator}
$I_{\al}$, $0<\al<n$, 
is defined by
\[
I_{\al}f(x)
:=
\int_{\R^n}\frac{f(y)}{|x-y|^{n-\al}}\,{\rm d}y,
\quad x\in\R^n,
\]
and the \textit{fractional maximal operator} 
$M_{\al}$, $0\le\al<n$, 
is defined by
\[
M_{\al}f(x)
:=
\sup_{x\in Q\in\cQ}
\ell_{Q}^{\al-n}
\int_{Q}|f(y)|\,{\rm d}y,
\quad x\in\R^n.
\]
Here, we use the notation 
$\cQ$ to denote the family of all cubes in $\R^n$ 
with sides parallel to the coordinate axes, 
$|Q|$ to denote the volume of $Q$ and 
$\ell_{Q}:=|Q|^{1/n}$.
When $\al=0$,
we simply write $M_0=M$ 
which is the classical Hardy-Littlewood maximal operator.

Let $0<p\le p_0<\8$. 
For an $L^p$ locally integrable function $f$ on $\R^n$ we set
\[
\|f\|_{\cM^{p,\,p_0}}
:=
\sup_{Q\in\cQ}
|Q|^{1/p_0}
\lt(\frac{1}{|Q|}\int_{Q}|f(x)|^p\,{\rm d}x\rt)^{1/p}.
\]
We will call the (classical) \textit{Morrey space} 
$\cM^{p,\,p_0}(\R^n)=\cM^{p,\,p_0}$ 
the subset of all $L^p$ locally integrable functions $f$ on $\R^n$ for which 
$\|f\|_{\cM^{p,\,p_0}}$ 
is finite. 
Applying H\"{o}lder's inequality,
we see that
\[
\|f\|_{\cM^{p_1,\,p_0}}
\ge
\|f\|_{\cM^{p_2,\,p_0}}
\quad\text{for all}\quad
p_0\ge p_1\ge p_2>0.
\]
This tells us that
\[
L^{p_0}=\cM^{p_0,\,p_0}
\subset\cM^{p_1,\,p_0}
\subset\cM^{p_2,\,p_0}
\quad\text{for all}\quad
p_0\ge p_1\ge p_2>0.
\]
Morrey spaces, 
introduced by C.~Morrey to study regularity questions
arising in the Calculus of Variations, 
describe local regularity more precisely than Lebesgue spaces 
and are widely used not only in harmonic analysis 
but also in partial differential equations 
(cf.~\cite{GT}).

The (well-known) Hardy-Littlewood-Sobolev theorem asserts that 
the Lebesgue space norm inequality
\[
\|I_{\al}f\|_{L^q(\R^n)}
\lesssim
\|f\|_{L^p(\R^n)}
\]
holds when 
$1<p<\frac{n}{\al}$ 
and $q=\frac{n}{n-\al p}p$.

In \cite{CF}, 
Chiarenza and Frasca showed that 
the Morrey space norm inequality
\[
\|I_{\al}f\|_{\cM^{q,\,q_0}(\R^n)}
\lesssim
\|f\|_{\cM^{p,\,p_0}(\R^n)}
\]
holds when 
$1<p\le p_0<\frac{n}{\al}$, 
$q=\frac{n}{n-\al p_0}p$ 
and 
$q_0=\frac{n}{n-\al p_0}p_0$.

Suppose that $\mu$ is a~nonnegative Radon measure on $\R^n$. 
Given any $\bt \in (0, n]$, define
\[
\|\mu\|_{\bt}
:=
\sup_{Q\in\cQ}
\frac{\mu(Q)}{\ell_{Q}^{\bt}}.
\]
Applying the Marcinkiewicz interpolation method, 
Adams in \cite{Ad} proved that 
the trace inequality 
\[
\|I_{\al}f\|_{L^q(\R^n,\,\mu)}
\lesssim
\|\mu\|_{n-\bt p}^{1/q}
\|f\|_{L^p(\R^n,\,{\rm d}x)}
\]
holds when
$1<p<\frac{n}{\al}$,
$0<\bt<\al<n$ and 
$q=\frac{n-\bt p}{n-\al p}p$.

The purpose of this paper is 
to extend this Adams trace inequality 
to product Morrey spaces.

Let $m\in\N$. 
For any vector-valued function
\[
\vec{f}
:=
(f_1,\ldots,f_m),
\]
with each $f_i$ being a locally integrable function in $\R^n$, 
and for any $x\in\R^n$, 
the \textit{$m$-linear fractional integral operator}
$\kI_{\al}$, $0<\al<mn$, 
is defined by
\[
\kI_\al(\vec{f})(x)
:=
\int_{(\R^n)^m}
\frac
{\prod_{j=1}^mf_j(y_j)}
{(|x-y_1|+|x-y_2|+\cdots+|x-y_m|)^{mn-\al}}
\,{\rm d}y_1{\rm d}y_2\cdots{\rm d}y_m,
\]
and the \textit{$m$-sublinear fractional maximal operator}
$\kM_{\al}$, $0\le \al<mn$, 
is defined by
\[
\kM_\al(\vec{f})(x)
:=
\sup_{x\in Q\in\cQ}
\ell_{Q}^{\al-mn}
\prod_{j=1}^m
\int_{Q}|f_j(y)|\,{\rm d}y.
\]
Given a~vector-valued exponent 
$\vec{P}:=(p_1,\ldots,p_m)$ 
with 
$p_1,\ldots,p_m\in(1, \8)$, let 
$1/p=1/p_1+\cdots+1/p_m$ and 
let $0<p\le p_0<\8$.
The \textit{product Morrey space} 
$\cM^{\vec{P},\,p_0}((\R^n)^m)$
is defined as the set of all 
$\vec{f}=(f_1,\ldots,f_m)$ 
satisfying
\begin{align*}
\|\vec{f}\|_{\cM^{\vec{P},\,p_0}((\R^n)^m)}
&:=
\sup_{Q\in\cQ}
|Q|^{1/p_0}
\prod_{j=1}^m
\lt(\frac{1}{|Q|}\int_{Q}|f_j(y)|^{p_j}\,{\rm d}y\rt)^{1/p_j}
\\ &=
\sup_{Q\in\cQ}
|Q|^{1/p_0-1/p}
\prod_{j=1}^m
\|f_j\|_{L^{p_j}(Q)}
<\8.
\end{align*}
Let $\mu$ be a~nonnegative Radon measure on $\R^n$. 
For $0<q\le q_0<\8$,
the \textit{Radon-Morrey space} 
$\cM_{\mu}^{q,\,q_0}(\R^n)$
consists of all $\mu$-measurable functions $f$ on $\R^n$ 
such that
\[
\|f\|_{\cM_{\mu}^{q,\,q_0}(\R^n)}
:=
\sup_{Q\in\cQ}
|Q|^{1/q_0-1/q}
\lt(\int_{Q}|f|^q\,{\rm d}\mu\rt)^{1/q}
<\8.
\]
The following are our main theorems, 
whose proofs are the focus of this paper.

\begin{thm}\label{thm1.1}
Let $\mu$ be a~nonnegative Radon measure on $\R^n$.
Let 
$\vec{P}:=(p_1,\ldots,p_m)$ with 
$p_1,\ldots,p_m\in(1, \8)$ and 
$1/p=1/p_1+\cdots+1/p_m$.
Assume that 
$1<p\le p_0<\frac{n}{\al}$ and 
$0<\bt<\al<mn$. Set 
$\tht=\frac{n-\bt p_0}{n-\al p_0}$,
$q=\tht p$ and $q_0=\tht p_0$.
Then we have
\[
\|\kI_\al(\vec{f})\|_{\cM_{\mu}^{q,\,q_0}(\R^n)}
\lesssim
\|\mu\|_{n-\bt p}^{1/q}
\|\vec{f}\|_{\cM^{\vec{P},\,p_0}((\R^n)^m)}.
\]
\end{thm}

\begin{thm}\label{thm1.2}
Let $\mu$ be a~nonnegative Radon measure on $\R^n$. 
Let 
$\vec{P}:=(p_1,\ldots,p_m)$ with 
$p_1,\ldots,p_m\in(1, \8)$ and 
$1/p=1/p_1+\cdots+1/p_m$.
Assume that 
$0<p\le p_0<\frac{n}{\al}$, 
$0<p\le 1$ and 
$0<\bt\le\al<mn$. Set 
$\tht=\frac{n-\bt p_0}{n-\al p_0}$,
$q=\tht p$ and $q_0=\tht p_0$.
Then we have
\[
\|\kI_\al(\vec{f})\|_{\cM_{\mu}^{q,\,q_0}(\R^n)}
\lesssim
\|\mu\|_{n-\bt p}^{1/q}
\|\vec{f}\|_{\cM^{\vec{P},\,p_0}((\R^n)^m)}.
\]
\end{thm}

\begin{xrem}
The $m$-linear fractional integral operator $\kI_{\al}$
and the $m$-sublinear fractional maximal operator $\kM_{\al}$ 
are studied on weighted Morrey type spaces 
in \cite{GM1,GM2,ISST1,ISST2,ISST3}.
\end{xrem}

The letter $C$ will be used for constants that may change from one occurrence to another.
Constants with subscripts, such as $C_1$, $C_2$, do not change in different occurrences.
By $A\approx B$ we mean that 
$c^{-1}B\le A\le cB$ 
with some positive finite constant $c$ independent of appropriate quantities. 
We write $X\lesssim Y$, $Y\gtrsim X$ 
if there is a independent constant $c$ such that $X \le cY$. 

\section{Proof of Theorems \ref{thm1.1} and \ref{thm1.2}}\label{sec2}
In what follows,
we prove Theorems \ref{thm1.1} and \ref{thm1.2}
using sparse technologies in weighted harmonic analysis.
By a standard argument,
it suffices to verify the theorems for dyadic cubes.

Denote by $\cD$ the family of all dyadic cubes 
\[
\cD
:=
\{2^{-k}(m+[0,1)^n):\,
k\in\Z,\,m\in\Z^n\}.
\]
Let $0<\eta<1$.
We say that a family 
$\cS\subset\cD$ is \textit{$\eta$-sparse} 
if for every $S\in\cS$, 
there exists a measurable set 
$E_{\cS}(S)\subset S$ such that 
$|E_{\cS}(S)|\ge\eta|S|$, 
and the sets 
$\{E_{\cS}(S):\,S\in\cS\}$ 
are pairwise disjoint. 

Let $m\in\N$ and $\al\in[0, mn)$. 
For any vector-valued function
$\vec{f}:=(f_1,\ldots,f_m)$
with each $f_i\ge 0$ being a~locally integrable function in $\R^n$,
it is known that the operators 
$\kM_{\al}$ and $\kI_{\al}$ 
admit the following discretizations:
for any $x\in\R^n$,
\begin{align*}
\kM_{\al}(\vec{f})(x)
&\lesssim
\sup_{Q\in\cD}
\ell_{Q}^{\al}
\prod_{j=1}^m
\frac{1}{|Q|}\int_{3Q}f_j(y)\,{\rm d}y\,
\1_{Q}(x),
\\
\kI_{\al}(\vec{f})(x)
&\lesssim
\sum_{Q\in\cD}
\ell_{Q}^{\al}
\prod_{j=1}^m
\frac{1}{|Q|}\int_{3Q}f_j(y)\,{\rm d}y\,
\1_{Q}(x),
\end{align*}
where 
$3Q$ denotes the cube with the same center as $Q$ 
with triple side length and 
$\1_{E}$ denotes the characteristic function 
of a~measurable set $E\subset\R^n$.
The first estimate follows from a simple geometric consideration for cubes, while the second one is proved in 
\cite[Section 4]{Moen09} and 
\cite[Lemma 2.1]{ISST3}.
Therefore, 
through the dyadic grid argument,
in order to prove Theorems \ref{thm1.1} and \ref{thm1.2}, 
it suffices to establish 
the estimates for the following 
dyadic discrete-type operators.
For any $x\in\R^n$, define
\begin{align*}
\kM_{\al}^{\cD}(\vec{f})(x)
&:=
\sup_{Q\in\cD}
\ell_{Q}^{\al}
\prod_{j=1}^m
\frac{1}{|Q|}\int_{Q}|f_j(y)|\,{\rm d}y\,
\1_{Q}(x)
\\ &=
\sup_{Q\in\cD}
\ell_{Q}^{\al-mn}
\prod_{j=1}^m
\int_{Q}|f_j(y)|\,{\rm d}y\,
\1_{Q}(x),
\\ \kI_{\al}^{\cD}(\vec{f})(x)
&:=
\sum_{Q\in\cD}
\ell_{Q}^{\al}
\prod_{j=1}^m
\frac{1}{|Q|}\int_{Q}f_j(y)\,{\rm d}y\,
\1_{Q}(x)
\\ &=
\sum_{Q\in\cD}
\ell_{Q}^{\al-mn}
\prod_{j=1}^m
\int_{Q}f_j(y)\,{\rm d}y\,
\1_{Q}(x).
\end{align*}

Given $Q\in\cD$ and 
$\cG\subset\cD$, we write
$\cG|_{Q}
:=
\{R\in\cG:\,R\subseteq Q\}$: 
the restriction to $Q$ of $\cG$.

\begin{thm}\label{thm2.1}
Let $\mu$ be a~nonnegative Radon measure on $\R^n$. 
Let $\vec{P}:=(p_1,\ldots,p_m)$ with 
$p_1,\ldots,p_m\in(1, \8)$ and 
$1/p=1/p_1+\cdots+1/p_m$.
Assume that 
$0<p\le p_0<\frac{n}{\al}$, 
$0<\al<mn$. 
Then we have
\[
\|\kM_{\al}^{\cD}(\vec{f})\|_{\cM_{\mu}^{p,\,p_0}(\R^n)}
\lesssim
\|\mu\|_{n-\al p}^{1/p}
\|\vec{f}\|_{\cM^{\vec{P},\,p_0}((\R^n)^m)}.
\]
\end{thm}

\noindent\textbf{Proof}\quad
We first observe that, 
for some appropriate $\eta$-sparse family
$\cS\subset\cD$,
\[
\kM_{\al}^{\cD}(\vec{f})
\lesssim
\sum_{S\in\cS}
|S|^{\al/n}
\prod_{j=1}^m
\frac{1}{|S|}\int_{S}|f_j(y)|\,{\rm d}y\,
\1_{E_{\cS}(S)},
\]
which can be verified by the argument in 
\cite[Proposition 3.8]{Cr}.

Fix $Q_0\in\cD$,
we decompose
\begin{align*}
\lefteqn{
\sum_{S\in\cS}
|S|^{\al/n}
\prod_{j=1}^m
\frac{1}{|S|}\int_{S}|f_j(y)|\,{\rm d}y\,
\1_{E_{\cS}(S)}
}\\ &=
\lt(
\sum_{S\in\cS|_{Q_0}}
+
\sum_{S\in\cS:\,S\supsetneq Q_0}
\rt)
|S|^{\al/n}
\prod_{j=1}^m
\frac{1}{|S|}\int_{S}|f_j(y)|\,{\rm d}y\,
\1_{E_{\cS}(S)}
\\ &=:
F_1+F_2.
\end{align*}

Since the sets 
$\{E_{\cS}(S):\,S\in\cS\}$ 
are pairwise disjoint,
\begin{align*}
\int_{Q_0}
F_1^p
\,{\rm d}\mu
&=
\sum_{S\in\cS|_{Q_0}}
\lt(
|S|^{\al/n}
\prod_{j=1}^m
\frac{1}{|S|}\int_{S}|f_j(y)|\,{\rm d}y
\rt)^p\mu(E_{\cS}(S))
\\ &\le
\sum_{S\in\cS|_{Q_0}}
\lt(
|S|^{\al/n}
\prod_{j=1}^m
\frac{1}{|S|}\int_{S}|f_j(y)|\,{\rm d}y
\rt)^p\mu(S)
\\ &\le
\|\mu\|_{n-\al p}
\sum_{S\in\cS|_{Q_0}}
\lt(
\prod_{j=1}^m
\frac{1}{|S|}\int_{S}|f_j(y)|\,{\rm d}y
\rt)^p|S|
\\ &\le
\frac{\|\mu\|_{n-\al p}}{\eta}
\sum_{S\in\cS|_{Q_0}}
\lt(
\prod_{j=1}^m
\frac{1}{|S|}\int_{S}|f_j(y)|\,{\rm d}y
\rt)^p|E_{\cS}(S)|.
\end{align*}
Applying the multiple H\"{o}lder's inequality 
with 
$p/p_1+\cdots+p/p_m=1$,
\begin{align*}
\lefteqn{
\le
\frac{\|\mu\|_{n-\al p}}{\eta}
\prod_{j=1}^m
\lt(
\sum_{S\in\cS|_{Q_0}}
\lt(\frac{1}{|S|}\int_{S}|f_j(y)|\,{\rm d}y\rt)^{p_j}
|E_{\cS}(S)|
\rt)^{p/p_j}} \\
&\le
\frac{\|\mu\|_{n-\al p}}{\eta}
\prod_{j=1}^m
\|M^{\cD}(f_j\1_{Q_0})\|_{L^{p_j}(Q_0)}^p
\\ &\lesssim
\frac{\|\mu\|_{n-\al p}}{\eta}
\prod_{j=1}^m
\|f_j\|_{L^{p_j}(Q_0)}^p,
\end{align*}
where we have used 
the $L^{p_j}$-boundedness of dyadic Hardy-Littlewood maximal operator 
$M^{\cD}$.
These yield
\begin{align*}
|Q_0|^{1/p_0-1/p}
\lt(
\int_{Q_0}
F_1^p
\,{\rm d}\mu\rt)^{1/p}
&\lesssim
\lt(\frac{\|\mu\|_{n-\al p}}{\eta}\rt)^{1/p}
|Q_0|^{1/p_0-1/p}
\prod_{j=1}^m
\|f_j\|_{L^{p_j}(Q_0)} 
\\ &\le
\lt(\frac{\|\mu\|_{n-\al p}}{\eta}\rt)^{1/p}
\|\vec{f}\|_{\cM^{\vec{P},\,p_0}((\R^n)^m)}.
\end{align*}

For $F_2$, we have that
\begin{align*}
\lefteqn{
|Q_0|^{1/p_0-1/p}
\lt(
\int_{Q_0}
F_2^p
\,{\rm d}\mu
\rt)^{1/p}}
\\ &=
|Q_0|^{1/p_0-1/p}
\lt[
\sum_{S\supsetneq Q_0}
\lt(
|S|^{\al/n}
\prod_{j=1}^m
\frac{1}{|S|}\int_{S}|f_j(y)|\,{\rm d}y
\rt)^p
\mu(Q_0\cap E_{\cS}(S))
\rt]^{1/p}
\\ &\le
|Q_0|^{1/p_0-1/p}
\lt[
\sum_{S\supsetneq Q_0}
\lt(
|S|^{\al/n}
\prod_{j=1}^m
\lt(
\frac{1}{|S|}\int_{S}|f_j(y)|^{p_j}\,{\rm d}y
\rt)^{1/p_j}
\rt)^p
\|\mu\|_{n-\al p}
|Q_0|^{1-\frac{\al p}{n}}
\rt]^{1/p}
\\ &\le
|Q_0|^{1/p_0-\al/n}
\|\mu\|_{n-\al p}^{1/p}
\|\vec{f}\|_{L^{\vec{P},\,p_0}((\R^n)^{m})}
\lt[
\sum_{S\supsetneq Q_0}
\lt(
|S|^{\al/n-1/p_0}
\rt)^p
\rt]^{1/p}
\\ &\lesssim
\|\mu\|_{n-\al p}^{1/p}
\|\vec{f}\|_{\cM^{\vec{P},\,p_0}((\R^n)^{m})},
\end{align*}
where we have used 
the formula for the sum of a geometric sequence 
with $\al/n-1/p_0<0$.
Taking supremum over all $Q_0\in\cD$ 
in the left-hand sides, 
we obtain the theorem.
\qed

\begin{thm}\label{thm2.2}
Let $\mu$ be a~nonnegative Radon measure on $\R^n$. 
Let 
$\vec{P}:=(p_1,\ldots,p_m)$ with 
$p_1,\ldots,p_m\in(1, \8)$ and 
$1/p=1/p_1+\cdots+1/p_m$.
Assume that 
$0<p\le p_0<\frac{n}{\al}$, 
$0<p\le 1$ and $0<\al<mn$. 
Then we have
\[
\|\kI_{\al}^{\cD}(\vec{f})\|_{\cM_{\mu}^{p,\,p_0}(\R^n)}
\lesssim
\|\mu\|_{n-\al p}^{1/p}
\|\vec{f}\|_{\cM^{\vec{P},\,p_0}((\R^n)^m)}.
\]
\end{thm}

\noindent\textbf{Proof}\quad
We observe that, 
for some appropriate $\eta$-sparse family
$\cS\subset\cD$,
\[
\kI_{\al}^{\cD}(\vec{f})
\lesssim
\sum_{S\in\cS}
|S|^{\al/n}
\prod_{j=1}^m
\frac{1}{|S|}\int_{S}|f_j(y)|\,{\rm d}y\,
\1_{S},
\]
which can be verified by the argument in 
\cite[Proposition 3.9]{Cr}.

Fix $Q_0\in\cD$,
we decompose
\begin{align*}
\sum_{S\in\cS}
|S|^{\al/n}
\prod_{j=1}^m
\frac{1}{|S|}\int_{S}|f_j(y)|\,{\rm d}y\,
\1_{S}
&=
\lt(
\sum_{S\in\cS|_{Q_0}}
+
\sum_{S\in\cS:\,S\supsetneq Q_0}
\rt)
|S|^{\al/n}
\prod_{j=1}^m
\frac{1}{|S|}\int_{S}|f_j(y)|\,{\rm d}y\,
\1_{S}
\\ &=:
F_1+F_2.
\end{align*}
Thanks to $0<p\le 1$, we have that 
\[
F_1^p
=
\lt(
\sum_{S\in\cS|_{Q_0}}
|S|^{\al/n}
\prod_{j=1}^m
\frac{1}{|S|}\int_{S}|f_j(y)|\,{\rm d}y\,
\1_{S}
\rt)^p
\le
\sum_{S\in\cS|_{Q_0}}
\lt(
|S|^{\al/n}
\prod_{j=1}^m
\frac{1}{|S|}\int_{S}|f_j(y)|\,{\rm d}y
\rt)^p
\1_{S}.
\]
Thus, the remainder of the proof 
is almost same as that of 
Theorem \ref{thm2.1}.
\qed

\vspace{2mm}

\noindent\textbf{Proof of Theorem \ref{thm1.1}}\quad
We use the Hedberg method (\cite{He}).
Assume that
$\|\vec{f}\|_{\cM^{\vec{P},\,p_0}((\R^n)^m)}=1$.
For any $0<\bt<\al<mn$ and 
$x\in Q\in\cD$,
by a~telescoping argument, 
one sees that
\[
|\kI_{\al}^{\cD}(\vec{f})(x)|
\lesssim
\ell_{Q}^{\al-\bt}
\kM_{\bt}^{\cD}(\vec{f})(x)
+
\ell_{Q}^{\al-n/p_0}.
\]
This yields 
\[
|\kI_{\al}^{\cD}(\vec{f})(x)|
\lesssim
\inf_{t>0}
\lt(
t^{\al-\bt}
\kM_{\bt}^{\cD}(\vec{f})(x)
+
t^{\al-n/p_0}
\rt).
\]
Solving the equation
\[
t^{\al-\bt}
\kM_{\bt}^{\cD}(\vec{f})(x)
=
t^{\al-n/p_0},
\]
we have that
\[
|\kI_{\al}^{\cD}(\vec{f})(x)|
\lesssim
\kM_{\bt}^{\cD}(\vec{f})(x)^{
\frac{n-\al p_0}{n-\bt p_0}
}
=
\kM_{\bt}^{\cD}(\vec{f})(x)^{1/\tht}.
\]
It follows that
\[
\|\kI_{\al}^{\cD}(\vec{f})\|_{\cM_{\mu}^{q,\,q_0}(\R^n)}
\lesssim
\|\kM_{\bt}^{\cD}(\vec{f})\|_{\cM_{\mu}^{p,\,p_0}(\R^n)}^{1/\tht}
\lesssim
\|\mu\|_{n-\bt p}^{1/q},
\]
where we have used Theorem \ref{thm2.1}.
\qed

\vspace{2mm}

\noindent\textbf{Proof of Theorem \ref{thm1.2}}\quad
Assume that
$\|\vec{f}\|_{\cM^{\vec{P},\,p_0}((\R^n)^m)}=1$.
For any $0<\bt\le\al<mn$ and 
$x\in Q\in\cD$, one sees that
\[
|\kI_{\al}^{\cD}(\vec{f})(x)|
\lesssim
\ell_{Q}^{\al-\bt}
|\kI_{\bt}^{\cD}(\vec{f})(x)|
+
\ell_{Q}^{\al-n/p_0}.
\]
This yields 
\[
|\kI_{\al}^{\cD}(\vec{f})(x)|
\lesssim
\inf_{t>0}
\lt(
t^{\al-\bt}
|\kI_{\bt}^{\cD}(\vec{f})(x)|
+
t^{\al-n/p_0}
\rt).
\]
Solving the equation
\[
t^{\al-\bt}
|\kI_{\bt}^{\cD}(\vec{f})(x)|
=
t^{\al-n/p_0},
\]
we have that
\[
|\kI_{\al}^{\cD}(\vec{f})(x)|
\lesssim
|\kI_{\bt}^{\cD}(\vec{f})(x)|^{
\frac{n-\al p_0}{n-\bt p_0}
}
=
|\kI_{\bt}^{\cD}(\vec{f})(x)|^{1/\tht}.
\]
It follows that
\[
\|\kI_{\al}^{\cD}(\vec{f})\|_{\cM_{\mu}^{q,\,q_0}(\R^n)}
\lesssim
\|\kI_{\bt}^{\cD}(\vec{f})\|_{\cM_{\mu}^{p,\,p_0}(\R^n)}^{1/\tht}
\lesssim
\|\mu\|_{n-\bt p}^{1/q},
\]
where we have used Theorem \ref{thm2.2}.
\qed

\subsection*{Ethics approval and consent to participate}
Not applicable.
\subsection*{Consent for publication}
Not applicable.
\subsection*{Competing Interests}
The authors have no competing interests to declare that are relevant to the content of this article.
\subsection*{Author contributions}
The authors contributed equally to the correctness of this paper.
\subsection*{Funding}
The first named author is financially supported by 
the Grant-in-Aid for JSPS Fellows (No. 25KJ0222).
The second named author is supported by 
Grant-in-Aid for Research Activity Start-up (25K23330)
The third named author is supported by 
Grant-in-Aid for Scientific Research (C) (23K03171), 
the Japan Society for the Promotion of Science. 
The forth named author is supported by 
Grant-in-Aid for Scientific Research (C) (15K04918 and 19K03510), 
the Japan Society for the Promotion of Science.


\begin{thebibliography}{999}

\bibitem{Ad} D.~R.~Adams, 
\emph{Traces of potentials arising from translation invariant operators},
Ann. Scuola Norm. Sup. Pisa Cl. Sci. (3) \textbf{25} (1971), 203--217.

\bibitem{CF}
F.~Chiarenza and M.~Frasca, 
\emph{Morrey spaces and Hardy-Littlewood maximal function}, 
Rend. Math. Appl., \textbf{7} (1987), 273--279.

\bibitem{Cr} D.~Cruz-Uribe, SFO,
\emph{Two weight norm inequalities for fractional integral operators and commutators},
arXiv:1412.4157v1 [math.CA] 12 Dec 2014.

\bibitem{GM1} L.~Grafakos and A.~Meskhi, 
\emph{Sharp Olsen's inequality for multilinear Riesz potentials}, 
Trans. A.~Razmadze Math. Inst. \textbf{175} (2021), no. 3, 433--436.

\bibitem{GM2} \bysame,
\emph{On sharp Olsen's and trace inequalities for multilinear fractional integrals}, 
Potential Anal. \textbf{59} (2023), no. 3, 1039--1050.

\bibitem{GT} D.~Gilbarg and S.~N.~Trudinger,
\emph{{\it Elliptic Partial Differential Equations of Second Order}},
2nd edn. (Springer, Berlin, 1983).

\bibitem{He} L.~I.~Hedberg, 
\emph{On certain convolution inequalities}, 
Proc. Amer. Math. Soc. \textbf{36} (1972), 505--510.

\bibitem{ISST1} T.~Iida, E.~Sato, Y.~Sawano and H.~Tanaka, 
\emph{Weighted norm inequalities for multilinear fractional operators on Morrey spaces},
Studia Math. \textbf{205} (2011), no. 2, 139--170.

\bibitem{ISST2} \bysame,
\emph{Multilinear fractional integrals on Morrey spaces}, 
Acta Math. Sin. (Engl. Ser.) \textbf{28} (2012), no. 7, 1375--1384.

\bibitem{ISST3} \bysame,
\emph{Sharp bounds for multilinear fractional integral operators on Morrey type spaces}, 
Positivity \textbf{16} (2012), no. 2, 339--358.

\bibitem{Moen09}
K.~Moen,
\emph{Weighted inequalities for multilinear fractional integral operators},
Collect. Math. \textbf{60}, 213--238 (2009).

\end{thebibliography}
\end{document}